\documentclass[twoside,reqno,10pt]{amsart}
\usepackage{amsmath, amssymb, amsaddr} 
\usepackage{url}
\newtheorem{theorem}{Theorem}[section]
\newtheorem{lemma}{Lemma}
\newtheorem{proposition}{Proposition}
\newtheorem{definition}{Definition}
\newtheorem{remark}{Remark}
\newtheorem{example}{Example}

\usepackage{tikz-cd}
\usetikzlibrary{matrix, calc, arrows}

\usepackage{color}
\definecolor{labelcolor}{RGB}{100,0,0}
\definecolor{outputcolor}{RGB}{0,0,100}

\usepackage{listings}	
\definecolor{dkgreen}{rgb}{0,0.6,0}
\definecolor{gray}{rgb}{0.5,0.5,0.5}
\definecolor{mauve}{rgb}{0.20 , 0.40, 1.0}

\lstdefinelanguage{Maxima}
{morekeywords={allbut, block, break, buildq, do, else, elseif, error, errcatch, for, go, if, in, is, local, new, step, then, thru, unless, while, return, true, false, tellsimp, tellsimpafter},
	sensitive=true,
	comment=[s]{/*}{*/},
	morestring=[b]"
}

\newcommand {\Cl}[1] {\ensuremath{ C\ell_{#1} } }

\begin{document}
	%
	\title[Multivector inverses in Clifford algebras]{Algorithmic computation of multivector inverses and characteristic polynomials in non-degenerate Clifford algebras}
	%
	%
	\author{Dimiter Prodanov\textsuperscript{1,2} }
	\address{
		PAML-LN, IICT, Bulgarian Academy of Sciences, Sofia, Bulgaria; \\
		Neuroelectronics Research Flanders and EHS, IMEC,  Leuven, Belgium 
		}
		\email{dimiter.prodanov@imec.be}
	
	\begin{abstract}
		The power of Clifford or, geometric, algebra lies in its ability to represent geometric operations in a concise and elegant manner.
		Clifford algebras provide the natural generalizations of complex, dual numbers and quaternions into non-commutative multivectors.
		The paper demonstrates an algorithm for the computation of inverses of such numbers in a non-degenerate Clifford algebra of an arbitrary dimension.
		The algorithm is a variation of the Faddeev-LeVerrier-Souriau algorithm and is implemented in the open-source Computer Algebra System Maxima.
		Symbolic and numerical examples in different Clifford algebras are presented. 
		
	\end{abstract}

	\maketitle
	
	\section{Introduction}\label{sec:intro}
	Clifford algebras provide the natural generalizations of complex, dual  and split-complex (or hyperbolic) numbers into the concept of \textit{Clifford numbers}, i.e. general multivectors.
	The power of Clifford or, geometric, algebra lies in its ability to represent geometric operations in a concise and elegant manner.
	The development of Clifford algebras is based on the insights of Hamilton, Grassmann, and Clifford from the 19\textsuperscript{th} century.
	After a hiatus lasting many decades, the Clifford geometric algebra experienced a renaissance with the advent of contemporary computer algebra systems. 
	Clifford algebras can be implemented in a variety of general-purpose computer languages and computational platforms.
	Recent years have seen renewed interest in Clifford algebra platforms: 
	notably, for Maple, Matlab, Mathematica, Maxima,
	\textit{Ganja.js} for JavaScript, \textit{GaLua} for Lua, 
	\textit{Galgebra} for Python, 
	\textit{Grassmann} for Julia. 
	
	Computation of  inverses of multivectors has drawn continuos attention in the literature as the problem was only gradually solved \cite{Hitzer2017,Acus2018,Hitzer2019,Shirokov2021}.
	The present contribution demonstrates an algorithm for  multivector inversion,
	 which involves only multiplications and subtractions and has a variable number of steps, depending on the spanning subspace of the multivector.  
	The algorithm is implemented in Maxima using the Clifford package \cite{Prodanov2016a,Prodanov2023}.

	In order to compute an inverse of a multivector, 
	previous contributions
	use series of automorphisms of special types. This allows one to write basis-free formulas with increasing complexity.
	On the other hand, the present algorithm is based on Faddeev--LeVerrier--Souriau (FVS) algorithm for characteristic polynomial and matrix inverse computation. 
	The correctness of the algorithm is proven using an algorithmic, constructive representation of a multivector in the matrix algebra over the reals, but it by no means depends on such a representation. 
	The present FVS algorithm is in fact a proof certificate for the existence of an inverse. 
	To the present author's knowledge the FVS algorithm has not been used systematically to exhibit multivector inverses.

	\section{Notation and Preliminaries}\label{sec:prelim}
	
	$\Cl{n}$ will denote a Clifford algebra of order \textit{n} but with unspecified signature. 
	Clifford multiplication is denoted by simple juxtaposition of symbols. 
	Algebra generators will be indexed by Latin letters. 
	Multi-indices will be considered as index \textbf{lists} and not as sets and will be denoted with capital letters.
	The operation of taking k-grade part of an expression will be denoted by $\left\langle . \right\rangle_k$ 
	and in particular the scalar part will be denoted by $\left\langle . \right\rangle_0$.
	Set difference is denoted by $\triangle$.
	Matrices will be indicated with bold capital letters, while matrix entries will be indicated by lowercase letters.
	The \textit{scalar product} of the blades will be denoted by $\ast$;
	$t$ in superscript will denote the grade negation operation, while $^\sim$ -- the Clifford product reversion. 
	
	\begin{definition}
		The generators of the Clifford algebra will be denoted by indexed symbol $e$.
		It will be assumed that there is an ordering relation $\prec$, such that for two natural numbers $i<j \Longrightarrow e_i \prec e_j$.
		The \textbf{extended basis} set of the algebra will be defined as the ordered power set  
		$
		\mathbf{B} := \left\lbrace P (E), {\prec} \right\rbrace 
		$	 
		of all generators $E =\{e_1, \ldots, e_n \}$ and their irreducible products. 
	\end{definition}
 
	\begin{definition}
		Define the diagonal scalar product matrix as 
		$
		\mathbf{G} := \{\sigma_{IJ}=e_I \ast e_J  | \ e_I, e_J \in  \mathbf{B}, I \prec J \} 
		$.
	\end{definition}
	A multivector will be written as $A= a_1 + \sum_{k=1}^{r} \left\langle A\right\rangle_k =  a_1 + \sum_{J} a_J e_J $. The maximal grade of $A$ will be denoted by $\mathrm{gr}[A]$. The pseudoscalar will be denoted by $I$.
	
	\section{Clifford algebra real matrix representation map}\label{sec:main2}
	
	In the present  we will focus on non-degenerate Clifford algebras, therefore the non-zero elements of $\mathbf{G}$ are valued in  $\left\lbrace -1, 1 \right\rbrace$.
	Supporting results are presented in Appendix \ref{sec:supp}.

	\begin{definition}[Clifford coefficient map]\label{def:coefmap}
		Define the linear map acting element-wise
		$
		C_a : \Cl{n} \mapsto \mathbb{R}  
		$
		by the action
		$ 	C_a (ax+b) = x$ 
		for $ x \in \mathbb{R}, a,b \in \mathbf{B}$.
		
		Define the Clifford \textbf{coefficient map} indexed by  $e_S$ as
		$
		\mathbf{A}_S :=C_S ( \mathbf{M})
		$,
		where $\mathbf{M} $ is the multiplication table of the extended basis $\mathbf{M}  = \left\lbrace   \mathcal{R} (e_M e_N) \ | \ e_M, e_N \in  \mathbf{B}  \right\rbrace$, and $\mathbf{A}_S$ action of the map.
	\end{definition}
	
	\begin{definition}[Canonical matrix map]
		\label{def:matcanon}
		Define the map $ \pi  :   \mathbf{B} \mapsto \mathbf{Mat}_\mathbb{R}  ({2^{n}} \times  {2^{n}} ), \ n=p+q+r	$ as
		$
		\pi : e_{S} \mapsto \mathbf{E}_s:= \mathbf{G} \mathbf{A}_s
		$
		where \textit{s} is the ordinal of $e_{S} \in \mathbf{B}$ and $\mathbf{A}_S$ is computed as in Def. \ref{def:coefmap}.
	\end{definition}
	
	\begin{proposition}\label{prop:coeflin}
		The  $\pi$-map is linear.
	\end{proposition}
	The proposition follows from the linearity of the coefficient map and matrix multiplication with a scalar.
	\begin{theorem}[Semigroup property]\label{th:semigr}
		Let $e_s$ and $e_t$ be generators of  $\Cl{p,q}$ .
		Then the following statements hold  
		\begin{enumerate}
			
			\item
			The map $\pi$ is a homomorphism  with regard to the Clifford product (i.e. $\pi$ distributes over the Clifford products):
			$
			\pi (e_s e_t) = \pi (e_s) \pi(e_t) 
			$.
			\item
			The set of all matrices $\mathbf{E}_s$ forms a multiplicative semigroup.
		\end{enumerate}
	\end{theorem}
	\begin{proof}
		Let $\mathbf{E}_s= \pi  (e_s),  \mathbf{E}_t = \pi( e_t), \mathbf{E}_{st}= \pi ( e_s e_t)$.
		We specialize the result of Lemma \ref{th:struct}  for $S= \{ s\}$ and $T= \{ t\}$ and observe that
		$
		m_{ \lambda \lambda^\prime  } \, e_{s t} =  m_{\lambda \mu }  \sigma_\mu m_{\mu \lambda^\prime }\, e_{s t} 
		$
		for  $ \lambda, \lambda^\prime, \mu  \leq n$ and 
		$
		\sigma_\lambda m_{ \lambda \lambda^\prime  } =  \sigma_\lambda m_{\lambda \mu }  \sigma_\mu m_{\mu \lambda^\prime } 
		$. In summary, 
		the map $\pi$ acts on \Cl{p,q} according to the following diagram:
		\begin{center}		
			\begin{tikzpicture}
				\matrix (m) [matrix of math nodes, row sep=3em,column sep=4em, minimum width=2em, minimum height=2em, nodes={asymmetrical rectangle}]
				{
					e_s &   \mathbf{E}_s   \\
					e_s e_t \equiv e_{st} &  \mathbf{E}_{st} \equiv \mathbf{E}_{s} \mathbf{E}_{t}, \ st = s \cup t \\
				};
				\path[->, font=\scriptsize ]
				(m-1-1) edge  node [right] { $e_t$  } (m-2-1)
				edge   node [above] {$\pi $} (m-1-2)
				(m-1-2) edge node [right] { $\mathbf{E}_t$  }(m-2-2)   
				(m-2-1) edge node[above] {$\pi$} (m-2-2);
			\end{tikzpicture}		
		\end{center}
		Therefore,
		$
		\mathbf{E}_{s t} = \mathbf{E}_s \mathbf{E}_t
		$.
		Moreover, we observe that 
		$
		\pi (e_s e_t) =\mathbf{E}_{s t}= \mathbf{E}_s \mathbf{E}_t = \pi (e_s) \pi(e_t) 
		$.
		
		For the semi-group property observe that 
		since $\pi$ is linear it is invertible. 
		Since $\pi$ distributes over Clifford product its inverse $\pi^{-1}$ distributes over
		matrix multiplication:
		\[
		\pi^{-1} (\mathbf{E}_{s} \mathbf{E}_{t} ) \equiv  \pi^{-1} (\mathbf{E}_{s t}) = e_{st} \equiv e_{s} e_{t} = \pi^{-1} (\mathbf{E}_s) \ \pi^{-1} (\mathbf{E}_t)
		\]
		However, $\Cl{p,q}$ is closed by construction, therefore, the set $\left\lbrace \mathbf{E}\right\rbrace _s$ is closed under	matrix multiplication.
	\end{proof}
	\begin{proposition}
		Let $\mathbf{L}:= \left\lbrace l_i | \; l_i \in \mathbf{B} \right\rbrace$ be a column vector and
		$\mathbf{R}_s$ be the first row of $\mathbf{E}_s$.
		Then 
		$\pi^{-1}: \mathbf{E}_s \mapsto  \mathbf{R}_s \mathbf{L}$.
	\end{proposition}
	\begin{proof}
		We observe that by the Prop. \ref{th:sparsec} the only non-zero element in the first row of $\mathbf{E}_s $ is 
		$\sigma_1 m_{1 s} = 1$.
		Therefore, 
		$ \mathbf{R}_s \mathbf{L} = e_s $.
	\end{proof}
	
	\begin{theorem}[Complete Real Matrix Representation]\label{th:clirep}
		Define the map
		$g: \mathbf{A} \mapsto \mathbf{G} \mathbf{A}$ as matrix multiplication with $\mathbf{G}$.
		Then for a fixed multiindex $s$ 
		$
		\pi = C_s \circ g = g \circ C_s
		$.
		Further, $\pi$ is an isomorphism inducing a Clifford algebra representation in the real matrix algebra:
		\begin{center}
			\begin{tikzpicture}
				\matrix (m) [matrix of math nodes, row sep=3em,column sep=4em, minimum width=2em, minimum height=2.5em, nodes={asymmetrical rectangle}]
				{
					\Cl{p,q} (\mathbb{R}) &   \mathbf{Mat}_{\mathbb{R}} \left( 2^{n} \times  2^{n}  \right)  \\
				};
				\path[ -> ]
				([yshift=0.5ex] m-1-1.east)  edge node [above] { $\pi$}   ([yshift=0.5ex] m-1-2.west) ;
				\path[ -> ]
				([yshift= -0.5ex] m-1-2.west)  edge   node [below] { $\pi^{-1}$} ([yshift=-0.5ex] m-1-1.east) ;
			\end{tikzpicture}
		\end{center}
	\end{theorem}
	\begin{proof}
		The  $\pi$-map is a linear isomorphism.
		The set $ \left\lbrace \mathbf{E}_s \right\rbrace  $ forms a multiplicative group, which is a subset of the matrix algebra 
		$ \mathbf{Mat}_{\mathbb{R}} (N \times N), N=2^n $. 
		Let 
		$ \pi (e_s)= \mathbf{E}_s$ and $
		\pi (e_t)= \mathbf{E}_t 
		$.
		It is claimed that
		\begin{enumerate}
			\item $ \mathbf{E}_s \mathbf{E}_t \neq \mathbf{0}$ by the Sparsity Lemma \ref{th:sparsity}.
			\item $ \mathbf{E}_s \mathbf{E}_t =- \mathbf{E}_t \mathbf{E}_s$ by Prop. \ref{prop:antisym}.
			\item $ \mathbf{E}_s \mathbf{E}_s = \sigma_s \mathbf{I}$ by Prop. \ref{prop:vsquare}.
		\end{enumerate}
		Therefore, the set $ \left\lbrace \mathbf{E}_S \right\rbrace_{S=\{1\}}^{P(n)} $ is an image of the extended basis $\mathbf{B}$. Here $P(n)$ denotes the power set of the indices of the algebra generators.
	\end{proof}
	What is special about the above representation is the relationship
	\begin{equation}\label{eq:trace}
		\mathrm{tr} \mathbf{A} = 2^n \left\langle A\right\rangle _0
	\end{equation}
	for the image $\pi(A)= \mathbf{A}$ of a general multivector element $A$
	and it will be used further in the proof of FVS algorithm.
	\begin{remark}
		The above construction works if instead of the entire algebra \Cl{p,q} we restrict a multivector to a sub-algebra of a smaller grade $\mathrm{max \ gr} [A] =r$. 
		In this case, we form grade-restricted multiplication matrices $\mathbf{G}_r$ and $\mathbf{M}_r$.
	\end{remark}
	
	\section{FVS multivector inversion algorithm}\label{sec:appli}
	Multivector inverses can be computed using the matrix representation and the characteristic polynomial.
	The matrix inverse is 
	$
	\mathbf{A}^{-1} =\hat{\mathbf{A}}/ {\det{\mathbf{A} }} 
	$,
	where $\hat{}$ denotes the  adjunct operation and $\det{\mathbf{A} }$ is the  determinant.
	The formula is not practical, because it requires the computation of  $n^2+1$ determinants. 
	By Cayley-Hamilton's Theorem, the inverse of \textbf{A} is a polynomial in \textbf{A}, which can be computed as the last step of the FVS algorithm \cite{Faddeev1949}.
	The  algorithm has a direct representation in terms of Clifford multiplications as follows:
	\begin{theorem}[Reduced-grade FVS algorithm]\label{th:clinverse}
		Suppose that $A \in \Cl{p,q}$ is a multivector of maximal grade $r \leq s$ and $ A \subseteq \mathrm{span}[e_1, \ldots, e_s]$.
		The Clifford inverse, if it exists, can be computed in 
		$k =  2^{\lceil s/2 \rceil} $ steps as
		\[
		\begin{array}{l|lr}
			m_1 = A               & c_1 = - k  {A}  \ast  1,    & t_1:= -c_1  \\
			m_2 = A m_2- t_1      & c_2 =-\frac{k}{2}   A \ast m_1,   & t_2:= -c_2  \\
			$\ldots$ 			  & $ \ldots $    \\
			m_k = A m_{k-1}- t_k  & c_k =   - A \ast m_{k-1},    & t_k:= -c_k
		\end{array}
		\]
		until the step where $m_k=0$ so that
		\begin{equation}
			{A}^{-1} = -{m_{k-1}}/{c_k}  .
		\end{equation}
		The inverse does not exist if $c_k= -\det{A} = 0$.
		
		The (reduced) characteristic polynomial of A of maximal grade $r$ is 
		\begin{equation}
			p_A( \lambda)=  \lambda^k + c_1 \lambda^{k-1} + \ldots c_{k-1} \lambda + c_k .
		\end{equation}
	\end{theorem}
	
	\begin{proof}
		The proof follows from the homomorphism of the $\pi$ map.
		We recall the statement of  FVS algorithm:
		\[
		p_A(\lambda)= \det{(\lambda \mathbf{I}_n  - \mathbf{A})}  = \lambda^n + c_1 \lambda^{n-1} + \ldots c_{n-1} \lambda + c_n, \quad n= \mathrm{dim}(\mathbf{A}),
		\]
		where
		\[
		\begin{array}{l|lr}
			\mathbf{M}_1 = \mathbf{A},                                   & t_1 = \ \  \mathrm{tr} [\mathbf{M}_1],            & c_1=- t_1 \\
			\mathbf{M}_2 = \mathbf{A} \mathbf{M}_1 - t_1 \mathbf{I}_n ,    & t_2 =\frac{1}{2}\mathrm{tr}[\mathbf{A} \mathbf{M}_1], & c_2=- t_2 \\
			$\ldots$ & $\ldots$ & $\ldots$ \\
			\mathbf{M}_n = \mathbf{A} \mathbf{M}_{n-1} - t_n \mathbf{I}_n, & t_n =\frac{1}{n}\mathrm{tr}[\mathbf{A} \mathbf{M}_{n-1}], & c_n=- t_n.
		\end{array}
		\]
		The matrix inverse can be computed from the last step of the algorithm as
		$
		\mathbf{A}^{-1} =   \mathbf{M}_{n-1} /{t_n}
		$
		under the obvious restriction $t_n \neq 0$.
		
		Therefore, for the k\textsuperscript{th} step of the algorithm application of $\pi^{-1}$ leads to
		$$
		\pi^{-1} :\mathbf{M}_k = \mathbf{A} \mathbf{M}_{k-1} - t_k  \mathbf{I} \mapsto	m_k = A m_{k-1}- t_k.
		$$ 
		Furthermore, 
		$  \mathrm{tr} [\mathbf{M}_k]  = n \left\langle m_k\right\rangle_0 = t_k $ by eq. \ref{eq:trace}.
		Moreover, the FVS algorithm  terminates with $\mathbf{M}_n=\mathbf{0}$, which corresponds to the limiting case $n=2^{p+q}$ wherever A contains all grades.	
		
		On the other hand, examining the matrix representations of different Clifford algebras, Acus and Dargys \cite{Acus2022} make the observation that according to the Bott periodicity the number of steps can be reduced to  $ 2^{\lceil n/2 \rceil}$.
		This can be proven as follows. 
		Consider the isomorphism
		$ 
		\Cl{p,q } \supset \Cl{p,q }^{+} \cong \Cl{q-1, p-1}
		$.
		Therefore, if a property holds for an algebra of dimension $n$ it will hold also for the algebra of dimension $n-2$. 
		Therefore, suppose that for $n$ even the characteristic polynomial is square free: $p_A(v) \neq q(v)^2$ for some polynomial. We proceed by reduction. 
		For $n=2$ in \Cl{2,0} and $A = {a_1}+{e_1} {a_2}+{e_2} {a_3}+{e_{12}} {a_4} $ we compute
		$
		p_A(v)= {{\left( {{a}_{1}^{2}}-{{a}_{2}^{2}}-{{a}_{3}^{2}}+{{a}_{4}^{2}}-2 {a_1} v+{{v}^{2}}\right) }^{2}}
		$
		and a similar result holds also for the other signatures of \Cl{2}. Therefore, we have a contradiction and the dimension can be reduced to $k=n/2$.
		
		In the same way, suppose that  $n$ is odd and the characteristic polynomial is square-free.
		However, for $n=3$ in \Cl{3,0} and 
		$A = {a_1}+{e_1} {a_2}+{e_2} {a_3}+{e_3} {a_4}+{a_5} e_{12} +{a_6} e_{13}  +{a_7} e_{23}  +{a_8}  e_{123} $
		it is established that $p_A(v)= q(v)^2$ for $	q(v) = $
		\begin{multline*}
		( {{a}_{1}^{2}}-{{a}_{2}^{2}}-{{a}_{3}^{2}}-{{a}_{4}^{2}}+{{a}_{5}^{2}} +{{a}_{6}^{2}}+{{a}_{7}^{2}}-{{a}_{8}^{2}} + 
		2 i ({a_3} {a_6}- {a_4} {a_5} - {a_2} {a_7}+ {a_1} {a_8} )-2 ({a_1} +  i {a_8}) v+{{v}^{2}}) \\
		( {{a}_{1}^{2}}-{{a}_{2}^{2}}-{{a}_{3}^{2}}-{{a}_{4}^{2}}  +{{a}_{5}^{2}}+{{a}_{6}^{2}}+{{a}_{7}^{2}}-{{a}_{8}^{2}}
		+2 i ({a_4} {a_5}- {a_3} {a_6}+ {a_2} {a_7}- {a_1} {a_8})-2( {a_1} -  i {a_8}) v+{{v}^{2}}). 
		\end{multline*}
		The above polynomial is factored due to space limitations. 
		Similar results hold also for the other signatures of \Cl{3}. 
		Therefore, we have a contradiction and the dimension can be reduced to $k=(n+1)/2$.
	Therefore, overall, one can reduce the number of steps to $k = 2^{\lceil n/2 \rceil} $.
	
	As a second case, suppose that $\mathrm{gr} [A]=r$.
	Let $E_r =\mathrm{span} [A]$  be the set of all generators, represented in A and \textit{s} their number.
	We compute the restricted multiplication tables $\mathbf{M} (E_r)$ and respectively $\mathbf{G} (E_r)$ and form the restricted map $\pi_r$. 
	Then 
	\[
	\pi_r (A A^{-1}) = \pi_r (A)  \pi_r(A^{-1}) =\mathbf{A}  \mathbf{A}^{-1} =\mathbf{I}_n, \quad n=2^s .
	\]
	Therefore, the FVS algorithm  terminates in $k=2^s$ steps.  
	Observe that 
	$
	\pi^{-1} :   \mathbf{A} \mathbf{M}_k  \mapsto A  m _k
	$. Therefore, $\mathrm{tr}[ \mathbf{A} \mathbf{M}_k ] $ will map to $n  A \ast m _k$ by eq. \ref{eq:trace}.
	Now, suppose that  $t_k \neq 0$; then for the last step of the algorithm we obtain:
	\[
	A m_{k-1}- t_k =0 \Rightarrow A  {m_{k-1}}/{t_k}=1 \Rightarrow A ^{-1}=	 {m_{k-1}}/{t_k} .
	\]
	Therefore, by the argument of the previous case, the number of steps can be  reduced to $k = 2^{\lceil s/2 \rceil} $.
\end{proof}
\section{Implementation}\label{sec:code}
Computations are performed using the \emph{Clifford} package in Maxima, which was first demonstrated in \cite{Prodanov2016a}. The present version of the package is 2.5 and it is available for download from a Zenodo repository \cite{Prodanov2023}.
The function \rmfamily{fadlevicg2cp} returns the inverse (if it exists) and the characteristic polynomial $p_A(v)$ of a multivector $A$ (Appendix \ref{sec:prg}).

\section{Experiments}

Experiments were performed on a Dell\textsuperscript{\textregistered } 64-bit Microsoft Windows 10 Enterprise machine with configuration -- 
Intel\textsuperscript{\textregistered } Core\textsuperscript{TM} i5-8350U CPU @ 1.70GHz, 1.90 GHz and 16GB RAM.
The computations were performed using the Clifford package version 2.5  
on Maxima version 5.46.0 using Steel Bank Common Lisp version 2.2.2.
\subsection{Symbolical experiments}\label{sec:symex}
\begin{example}
	For \Cl{2,0} and a multivector $ A= a_0 + a_1 e_1 +a_2 e_2 + a_3 e_{12} $
	the  reduced grade algorithm produces
	\[
	\begin{array}{ll}
		t_{1}=-2 a_1, &
		m_{1}=a_1+e_1 a_2+e_2 a_3+a_4 e_{12}, \\
	\end{array}
	\]
	resulting in
	$
	A^{-1} = {({a_1}-{e_1} {a_2}-{e_2} {a_3}-{a_4}  e_{12}) }/{({{a}_{1}^{2}}-{{a}_{2}^{2}}-{{a}_{3}^{2}}+{{a}_{4}^{2}})}
	$
	and the characteristic polynomial is
	$
	p_A(v)=
	{{a}_{1}^{2}}-{{a}_{2}^{2}}-{{a}_{3}^{2}}+{{a}_{4}^{2}}-2 {a_1} v+{{v}^{2}}
	$.
\end{example}
\begin{example}
	For \Cl{1,1} and a multivector$ A= a_0 + a_1 e_1 +a_2 e_2 + a_3 e_{12} $
	the reduced grade algorithm produces
	\[
	\begin{array}{ll}
		t_{1}=-2 a_1, &
		m_{1}=a_1+e_1 a_2+e_2 a_3+a_4 e_{12},\\
	\end{array}
	\]
	resulting in
	$
	A^{-1} =  {(-{a_1}+{e_1} {a_2}+{e_2} {a_3}+{a_4}  e_{12} )}/{(-{{a}_{1}^{2}}+{{a}_{2}^{2}}-{{a}_{3}^{2}}+{{a}_{4}^{2}})}
	$
	and the characteristic polynomial is 
	$
	p_A(v)=
	{{a}_{1}^{2}}-{{a}_{2}^{2}}+{{a}_{3}^{2}}-{{a}_{4}^{2}}-2 {a_1} v+{{v}^{2}}
	$.
\end{example}
\begin{example}
	For \Cl{0,2} and a multivector $ A= a_0 + a_1 e_1 +a_2 e_2 + a_3 e_{12} $
	the reduced grade algorithm produces
	\[
	\begin{array}{ll}
		t_{1}=-2 a_1, &
		m_{1}=a_1+e_1 a_2+e_2 a_3+a_4 e_{12},\\
	\end{array}
	\]
	resulting in
	$
	A^{-1} = {({a_1}-{e_1} {a_2}-{e_2} {a_3}-{a_4}  e_{12}) }/{({{a}_{1}^{2}}+{{a}_{2}^{2}}+{{a}_{3}^{2}}+{{a}_{4}^{2}})}
	$
	and the characteristic polynomial is
	$
	p_A(v)=
	{{a}_{1}^{2}}+{{a}_{2}^{2}}+{{a}_{3}^{2}}+{{a}_{4}^{2}}-2 {a_1} v+{{v}^{2}}
	$.
\end{example}
Bespoke computations are practically instantaneous on the testing hardware configuration.
Higher-dimensional symbolic examples produce very long expressions and are not particularly instructive.  

\subsection{Numerical experiments} 
Note that the trivial last steps will be omitted because of space limitations.
To demonstrate the utility of FVS algorithm here follow some high-dimensional numerical examples.
\begin{example}
Let us compute a rational example in \Cl{2,5}.	
Let $A=1 - 2 B+5 C   $, where $B:=e_{15} $ and $C:= e_1e_3e_4$. Then $ \mathrm{span}[A]= \{e_1, e_3, e_4, e_5 \}$ and
 for the maximal representation we have $k=2^4=16$ steps:
\[  
\begin{array}{ll}
t_{ 1 }= -16  , & m_{ 1 }= -15+5 C-2 B; \\ 
t_{ 2 }= 288  , & m_{ 2 }= 252-70 C+28 B; \\ 
t_{ 3 }= -2912  , & m_{ 3 }= -2366+1190 C-476 B; \\ 
t_{ 4 }= 29456  , & m_{ 4 }= 22092-10640 C+4256 B; \\ 
t_{ 5 }= -213696  , & m_{ 5 }= -146916+99820 C-39928 B; \\ 
t_{ 6 }= 1509760  , & m_{ 6 }= 943600-634760 C+253904 B; \\ 
t_{ 7 }= -8250496  , & m_{ 7 }= -4640904+4083240 C-1633296 B; \\ 
t_{ 8 }= 43581024  , & m_{ 8 }= 21790512-19121280 C+7648512 B; \\ 
t_{ 9 }= -181510912  , & m_{ 9 }= -79411024+89831280 C-35932512 B; \\ 
t_{ 10 }= 730723840  , & m_{ 10 }= 274021440-307223840 C+122889536 B; \\ 
t_{ 11 }= -2275435008  , & m_{ 11 }= -711073440+1062883360 C-425153344 B; \\ 
t_{ 12 }= 6900244736  , & m_{ 12 }= 1725061184-2492483840 C+996993536 B; \\ 
t_{ 13 }= -15007376384  , & m_{ 13 }= -2813883072+6132822080 C-2453128832 B; \\ 
t_{ 14 }= 32653412352  , & m_{ 14 }= 4081676544-7936593280 C+3174637312 B; \\ 
t_{ 15 }= -39909726208  , & m_{ 15 }= -2494357888+12471789440 C-4988715776 B. \\
\end{array}
\]
Therefore,
$
A^{-1}=  \left(1 -5 C + 2 B   \right) /22
$
and
$
p_A(v)= (22 -2 v +v^2)^8
$.
Evaluation took 0.0469 s using 12.029 MB memory on Maxima.
On the other hand, the reduced algorithm will run in $k=2^{ \lceil 4/2 \rceil } =4$ steps:
\[  
\begin{array}{ll}
t_{1}=-4, & m_{1}=1+5 C-2 B;\\
t_{2}=48, & m_{2}=-24-10 C+4 B;\\
t_{3}=-88, & m_{3}=66+110 C-44 B; 
\end{array}
\]
and $p_A(v) =484-88v+48v^2-4v^3+v^4= (22 -2 v +v^2)^2$. 
Evaluation took 0.0156 s  using 2.512 MB memory on Maxima.
Note, that in this case
$ det A = A A^{\sim} = 22$. Therefore, in accordance with Shirokov's approach
$
A^{-1}=    A^{\sim}/22
$. 
\end{example}

\begin{example}
Consider \Cl{5,2} and  let
$
A= 1 -{e_2} + I
$.
The full-grade algorithm takes 128 steps and will not be illustrated due to space limitation.
The reduced grade algorithm can be illustrated as follows.
Let $C=e_{134567}$. Then  
\[
\begin{array}{ll}
t_{1}=-16, & m_{1}=1-e_{2}+I;\\
t_{2}=120, & m_{2}=-15+14 e_{2}-14 I+2 C;\\
t_{3}=-560, & m_{3}=105-89 e_{2}+93 I-26 C;\\
t_{4}=1836, &m_{4}=-459+340 e_{2}-388 I+156 C;\\
t_{5}=-4560, & m_{5}=1425-881 e_{2}+1145 I-572 C;\\
t_{6}=9064, & m_{6}=-3399+1682 e_{2}-2562 I+1454 C;\\
t_{7}=-14960, & m_{7}=6545-2529 e_{2}+4557 I-2790 C;\\
t_{8}=20886, & m_{8}=-10443+3096 e_{2}-6648 I+4296 C;\\
t_{9}=-24880, & m_{9}=13995-3051 e_{2}+8091 I-5448 C;\\
t_{10}=25480, & m_{10}=-15925+2386 e_{2}-8242 I+5694 C;\\
t_{11}=-22416, & m_{11}=15411-1475 e_{2}+7007 I-4934 C;\\
t_{12}=16716, & m_{12}=-12537+596 e_{2}-4932 I+3548 C;\\
t_{13}=-10480, & m_{13}=8515-35 e_{2}+2795 I-1980 C;\\
t_{14}=5400, & m_{14}=-4725-50 e_{2}-1150 I+850 C;\\
t_{15}=-2000, & m_{15}=1875+125 e_{2}+375 I-250 C,\\
\end{array}
\]
resulting in
$
A^{-1}=  ({1} -   {e_2}- {3} \,  I + {2}\,  C )/5.
$.
The characteristic polynomial can factorize as
$
p_A(v) =(1+v^2)^4 (5-4 v+v^2)^4
$.

It should be noted that in this case, the determinant $ \det{A}$ can be computed by the sequence of operations 
$
B= A A^t = 1- 2 I
$, followed by
$
\det{A}= B B^{\sim} = 5
$.  
This allows for writing the formula
\[
A^{-1} =  A^t   (A A^t )^\sim/5
\]
in accordance with Shirokov's approach. 
\end{example}

\section{Concluding remarks}\label{sec:conclusion}
The maximal matrix algebra construction exhibited in the present paper allows for systematic translation of matrix-based algorithms to Clifford algebra simultaneously allowing for their direct verification. 

The advantage of the multivector FVS algorithm is its simplicity of implementation. 
This can be beneficial for purely numerical applications as it involves only Clifford  multiplication, and taking scalar parts of multivectors, which can be encoded as the first member of an array. 
The Clifford multiplication computation can be reduced to $  \mathcal{O} (N \log{N})$ operations, since it involves sorting of a joined list of algebra generators.  
On the other hand, the present algorithm does not ensure optimality of the computation but only provides a certificate of existence of an inverse. 
Therefore, optimized algorithms can be introduced for particular applications, i.e.  Space-Time Algebra \Cl{1,4}, Projective Geometric Algebra \Cl{3,0,1}, Conformal Geometric Algebra \Cl{4,1}, etc. 
As a side product, the algorithm can compute the characteristic polynomial of a general multivector and, hence,  its determinant also without any resort to a matrix representation. 
This could be used, for example, for  computation of a multivector resolvent or some other analytical functions. 
	
One of the main applications of the present algorithms could be in Finite Element Modelling where the Geometric algebra can improve the efficiency and accuracy of calculations by providing a more compact representation of vectors, tensors, and geometric operations. 
This can lead to faster and more accurate simulations of elastic deformations. 


\section*{Acknowledgment}
The present work is funded  in part by the European Union's Horizon Europe program under grant agreement VIBraTE, grant agreement 101086815.

\appendix

\section{Supporting results}\label{sec:supp}

\begin{definition}[Sparsity property]
A matrix has the \textit{sparsity property} if it has exactly one non-zero element per column and exactly one non-zero element per row. Such a matrix we call sparse.
\end{definition}

\begin{lemma}[Sparsity lemma]\label{th:sparsity}
If the matrices $\mathbf{A}$ and $\mathbf{B}$ are sparse then so is $\mathbf{C}=\mathbf{A B}$.
Moreover,
\[
c_{ij} = \begin{cases}
0 \\
a_{i q} b_{q j} 
\end{cases}
\] 
(no summation!) for some index $q$.
\end{lemma}
\begin{proof}
Consider two sparse square matrices $\mathbf{A}$ and $\mathbf{B}$  of dimension $n$.
Let $c_{ij} = \sum_{\mu} a_{i \mu} b_{\mu j}$. Then
as we vary the row  index $i$ then there is only one index $ q \leq n$, such that $ a_{i q}  \neq 0$.
As we vary the column index $j$ then there is only one index  $ q \leq n$, such that $ b_{ q j}  \neq 0$.
Therefore, $c_{ij} = (0; a_{i q} b_{q j}  ) $
for some $q$ by the sparsity of $\mathbf{A}$ and $\mathbf{B}$.	
As we vary the row  index $i$ then $c_{q j} =0 $ for $i \neq q$ for the column $j$ by the sparsity of  $\mathbf{A}$.
As we vary the column index $j$ then $c_{i q}=0$ for $j \neq q$ for the row $i$ by the sparsity of  $\mathbf{B}$.
Therefore, $\mathbf{A B}$ is sparse. 
\end{proof}

\begin{lemma}[Multiplication Matrix Structure]\label{th:struct}
For the  multi-index  disjoint sets $S \prec T$ 
the following implications hold for the   elements of $\mathbf{M}$ :
\begin{center}

\begin{tikzpicture}
	\matrix (m) [matrix of math nodes,row sep=3em,column sep=4em,minimum width=2em]
	{
		m_{\mu \lambda}\,  e_S &  m_{\mu   \lambda^\prime} e_T & \\
		m_{\lambda \mu} e_S &  m_{\lambda \mu} m_{\mu \lambda^{\prime } } e_{S \triangle T} &   m_{\lambda \lambda^{\prime \prime} } e_{S \triangle T} \\};
	\path[-> ]
	(m-1-1) edge node [left] { $\exists  $} (m-2-1)
	edge node [above] {$\exists \lambda^\prime > \lambda $} (m-1-2)
	(m-1-2) edge (m-2-2)
	(m-2-1) edge node[above] {$\exists$} (m-2-2)
	(m-2-2) edge node[above]{$\exists \lambda^{\prime \prime}= \lambda^\prime $}  (m-2-3);
\end{tikzpicture} 
\end{center}
so that 
$
m_{ \lambda \lambda^\prime  }  =  m_{\lambda \mu }  \sigma_\mu m_{\mu \lambda^\prime }
$
for some index $\mu$.
\end{lemma}

\begin{proof}
Suppose that the ordering of elements is given in the construction of \Cl{p,q,r}.
To simplify presentation, without loss of generality, suppose that $e_s$ and $e_t$ are some generators.
By the properties of $\mathbf{M}$ there exists an index $\lambda^\prime > \lambda$, such that
$e_{M} e_{L^\prime}   = m_{\mu \lambda^\prime }\,  e_t$, $  L^\prime \backslash M  = T$
for $L \prec   L^\prime $.
Choose $M$, s.d. $L \prec M \prec  L^\prime $.
Then for
$L \prec M \prec  L^\prime $ and $S \prec T$
\begin{flalign*}
e_{M} e_{L} & = m_{\mu \lambda}\,  e_s , \ \ {\small L \triangle M } = S
\Leftrightarrow e_{L} e_{M} = m_{\lambda \mu }\,  e_s \\
e_{M} e_{L^\prime} & = m_{\mu \lambda^\prime }\,  e_t, \ \ {\small L^\prime \triangle M } = T
\end{flalign*}

Suppose that $e_s e_t =  e_{s t}, st = S \cup T = S \triangle T$. 
Multiply together the diagonal nodes in the matrix
\[
e_{L} \underbrace{e_{M}  e_{M}}_{\sigma_\mu} e_{L^\prime} = m_{\lambda \mu }  m_{\mu \lambda^\prime }\, e_{s t} 
\]
Therefore,
$ s\in L $ and
$ t \in L^\prime $.
We observe that there is at least one element (the algebra unity) with the desired property $\sigma_\mu \neq 0$.

Further, we observe that there exists unique index $\lambda^{\prime \prime} $ such that $  m_{\lambda \lambda^{\prime \prime} } e_{s t} $.
Since $\lambda $ is fixed. This implies that $ L^{\prime \prime} = L^\prime \Rightarrow \lambda^{\prime \prime} = \lambda^{\prime}$.
Therefore,  
\[
e_{L} e_{L^\prime} = m_{ \lambda \lambda^\prime  } e_{s t},  \ \  L^\prime \triangle L  = \{ s, t \} 
\]
which implies the identity
$
m_{ \lambda \lambda^\prime  } \, e_{s t} =  m_{\lambda \mu }  \sigma_\mu m_{\mu \lambda^\prime }\, e_{s t} 
$.
For higher graded elements $e_S$ and $e_T$ we should write $e_{S \triangle T}$ instead of $e_{st}$.

\end{proof}

\begin{proposition}\label{prop:msparse}
Consider the multiplication table $\mathbf{M}$.
All elements $m_{k  j} $ are different for a fixed row $k$.
All elements $m_{i q }$ are different for a fixed column $q$.
\end{proposition}
\begin{proof}
Fix $k$. Then for $e_K, e_J \in \mathbf{B} $ we have
$
e_{K} e_{J} = m_{k  j} e_S, \ \ S = K \triangle J
$.
Suppose that we have equality for 2 indices $j, j^\prime$. Then
$
K \triangle J^\prime = K \triangle J = S
$.
Let $\delta = J \cap J^\prime$; then 
\[
K \triangle  \left( J \cup \delta \right)  = K \triangle J = S \Rightarrow K \triangle \delta = S \Rightarrow \delta = \emptyset
\]
Therefore, $j = j^\prime$. 
By symmetry, the same reasoning applies to a fixed column $q$.
\end{proof}
\begin{proposition}\label{th:sparsec}
For $e_s \in \mathbf{E}$ the matrix $\mathbf{A}_s = C_{s} (\mathbf{M})$ is sparse. 
\end{proposition}
\begin{proof}
Fix an element $e_s \in \mathbf{E}$.
Consider a row $k$.
By Prop. \ref{prop:msparse} there is a $j$, such $e_{k j}= e_s$.
Then $a_{k j}= m_{k j}$, while for $i \neq j$ $a_{k i}=0$.

Consider a column $l$
By Prop. \ref{prop:msparse} there is a $j$, such $e_{j l}= e_s$.
Then $a_{j l}= m_{j l}$, while for $i \neq j$ $a_{i l}= 0$.
Therefore, $\mathbf{A}_s$ has the sparsity property.
\end{proof}

\begin{proposition}\label{prop:antisym}
For generator elements $e_s$ and $e_t$
$
\mathbf{E}_s \mathbf{E}_t + \mathbf{E}_t \mathbf{E}_s= \mathbf{0}
$. 
\end{proposition}
\begin{proof}
Consider the basis elements $e_s$ and $e_t$. 
By linearity and homomorphism of the $\pi$ map (Th. \ref{th:semigr}):
$
\pi: e_s e_t + e_t e_s = 0 \mapsto \pi(e_s e_t) + \pi(e_t e_s) = \mathbf{0}
$.
Therefore, for two vector elements
$
\mathbf{E}_s \mathbf{E}_t + \mathbf{E}_t \mathbf{E}_s= \mathbf{0}
$.
\end{proof}

\begin{proposition}\label{prop:vsquare}
$ \mathbf{E_s} \mathbf{E_s} = \sigma_s \mathbf{I} $ 
\end{proposition}
\begin{proof}
Consider the matrix
$ \mathbf{W} = \mathbf{G}  \mathbf{A_s} \mathbf{G}  \mathbf{A_s}   $. 
Then 
$ w_{\mu \nu} = \sum_{\lambda} \sigma_\mu \sigma_\lambda a_{ \mu \lambda } a_{  \lambda \nu}$  element-wise.
By Lemma \ref{th:sparsity} $\mathbf{W} $ is sparse so that
$w_{\mu \nu} =  (0; \sigma_\mu \sigma_q a_{ \mu q } a_{  q \nu}) $.

From the structure of $\mathbf{M}$ for the entries containing the element $e_S$ we have the equivalence
\[
\begin{cases}
e_{M} e_{Q} = a^s_{\mu q} e_S,  \ \  S= M \triangle Q \\
e_{Q} e_{M} = a^s_{q \mu} e_S, 
\end{cases}	
\]
After multiplication of the equations we obtain
$
e_{M} e_{Q}  e_{Q} e_{M} = a^s_{\mu q} e_S a^s_{q \mu} e_S 
$,
which simplifies to the \textit{First fundamental identity}:
\begin{equation}\label{eq:fistId}
\sigma_q \sigma_\mu = a^s_{\mu q}   a^s_{q \mu} \sigma_s
\end{equation}
We observe that if  $\sigma_\mu=0$ or $\sigma_q =0$ the result follows trivially.
In this case also $\sigma_s=0$.
Therefore, let's suppose that $ \sigma_s \sigma_q \sigma_\mu \neq 0 $.
We multiply both sides by $ \sigma_s \sigma_q \sigma_\mu  $
to obtain $
\sigma_s = \sigma_q \sigma_\mu a^s_{\mu q}   a^s_{q \mu} 
$.
However, the RHS is a diagonal element of $\mathbf{W}$, therefore by the sparsity it is the only non-zero element for a given row/column so that 
$
\mathbf{W}= \mathbf{E_s^2} = \sigma_s \mathbf{I}
$. 
\end{proof}

\section{Program code}\label{sec:prg}
The Clifford package can be downloaded from a Zenodo repository \cite{Prodanov2023}.
The examples can be downloaded from a Zenodo repository  and it includes the file \rmfamily{climatrep.mac}, which implements different instances of the FVS algorithm \cite{Prodanov2023a}.
\begin{lstlisting}[caption={FVS Maxima code based on the Clifford package}, label=lst:rep]
	fadlevicg2cp(A, v):=block(
	[M:1, K, i:1, n, k:length(clv(A)), cq, c, ss],
		n:2^(ceiling(k/2)),
		array(c,n+1),  for r:0 thru n+1 do c[r]:1,
		A:rat(A),  
		ss:c[1]*v^^n,
		while i<n and K#0 do (
			K:dotsimpc(expand (A.M)),
			cq:-n/i*scalarpart(K),
			if _debug1=all then print("t_{",i,"}=",cq," m_{",i,"}=",K,"\\\\"),
			if K#0 then
				M: rat(K + cq), 
			c[i+1]:cq, ss:ss+c[i+1]*v^^(n-i),
			i:i+1
		), 
		K:dotsimpc(expand(A.M)),
		cq:-n/i*scalarpart(K),
		if _debug1=all then print("t_{",i,"}=",cq," m_{",i,"}=",K,"\\\\"),
		ss:ss+cq,
		if cq=0 then cq:1, 	M:factor(-(M)/cq),
		[M, ss]
	);
\end{lstlisting}
%
%
%
\bibliographystyle{splncs04}
\bibliography{clibib1}

\begin{thebibliography}{1}
\providecommand{\url}[1]{\texttt{#1}}
\providecommand{\urlprefix}{URL }
\providecommand{\doi}[1]{https://doi.org/#1}

\bibitem{Acus2018}
Acus, A., Dargys, A.: The inverse of a multivector: Beyond the threshold $p+q =
  5$. Adv. Appl. Clifford Algebras  \textbf{28}(3) (2018).
  \doi{10.1007/s00006-018-0885-4}

\bibitem{Acus2022}
Acus, A., Dargys, A.: The characteristic polynomial in calculation of
  exponential and elementary functions in {Clifford} algebras. Math. Methods
  Appl. Sci.  (2022). \doi{10.22541/au.167101043.33855504/v1}

\bibitem{Faddeev1949}
Faddeev, D.K., Sominskij, I.S.: Sbornik Zadatch po Vyshej Algebre. Nauka,
  Moscow--Leningrad (1949)

\bibitem{Hitzer2019}
Hitzer, E., Sangwine, S.J.: Construction of multivector inverse for {Clifford}
  algebras over $ 2 m + 1$ -- dimensional vector spaces from multivector
  inverse for clifford algebras over 2m-dimensional vector spaces. Adv. Appl.
  Clifford Algebras  \textbf{29}(2) (2019). \doi{10.1007/s00006-019-0942-7}

\bibitem{Hitzer2017}
Hitzer, E., Sangwine, S.: Multivector and multivector matrix inverses in real
  clifford algebras  \textbf{311},  375--389. \doi{10.1016/j.amc.2017.05.027}

\bibitem{Prodanov2023}
Prodanov, D.: Clifford {Maxima} package v 2.5.4. \doi{10.5281/ZENODO.8205828},
  \url{https://zenodo.org/record/8205828}

\bibitem{Prodanov2023a}
Prodanov, D.: Examples for {CGI2023}. \doi{10.5281/ZENODO.8207889}

\bibitem{Prodanov2016a}
Prodanov, D., Toth, V.T.: Sparse representations of {Clifford} and tensor
  algebras in {Maxima}. Adv. Appl. Clifford Algebras pp. 1--23 (2016).
  \doi{10.1007/s00006-016-0682-x}

\bibitem{Shirokov2021}
Shirokov, D.S.: On computing the determinant, other characteristic polynomial
  coefficients, and inverse in {Clifford} algebras of arbitrary dimension.
  Comp. Appl. Math.  \textbf{40}(5) (2021). \doi{10.1007/s40314-021-01536-0}

\end{thebibliography}
\end{document}